\documentclass{amsart} 

\usepackage{tikz}
\tikzset{every picture/.style={line width=0.75pt}} 

\usepackage{mathtools}

\usepackage[usenames,dvipsnames]{pstricks}
\usepackage{epsfig}
\usepackage{pst-grad} 
\usepackage{pst-plot} 
\usepackage[space]{grffile} 
\usepackage{etoolbox} 
\makeatletter 
\patchcmd\Gread@eps{\@inputcheck#1 }{\@inputcheck"#1"\relax}{}{}
\makeatother

\usepackage[T1]{fontenc}

\usepackage{amsmath}
\usepackage{amssymb}
\usepackage{amsthm}
\usepackage[latin1]{inputenc}
     
\usepackage{graphicx}

\usepackage[usenames,dvipsnames]{pstricks}
\usepackage{epsfig}
\usepackage{pst-grad} 
\usepackage{pst-plot} 

\usepackage{ifthen}
\usepackage{color}

\newcommand{\intav}[1]{\mathchoice {\mathop{\vrule width 6pt height 3 pt depth  -2.5pt
\kern -8pt \intop}\nolimits_{\kern -6pt#1}} {\mathop{\vrule width
5pt height 3  pt depth -2.6pt \kern -6pt \intop}\nolimits_{#1}}
{\mathop{\vrule width 5pt height 3 pt depth -2.6pt \kern -6pt
\intop}\nolimits_{#1}} {\mathop{\vrule width 5pt height 3 pt depth
-2.6pt \kern -6pt \intop}\nolimits_{#1}}}

\def\polhk#1{\setbox0=\hbox{#1}{\ooalign{\hidewidth\lower1.5ex\hbox{`}\hidewidth\crcr\unhbox0}}}

\def\XXint#1#2#3{{\setbox0=\hbox{$#1{#2#3}{\int}$ }
\vcenter{\hbox{$#2#3$ }}\kern-.6\wd0}}

\usepackage{setspace}

\newtheorem{theorem}{Theorem}

\newtheorem{definition}{Definition}

\newtheorem{proposition}{Proposition}

\newtheorem{Assumption}{Assumption}

\makeatletter
\@namedef{subjclassname@2020}{%
  \textup{2020} Mathematics Subject Classification}
\makeatother

\makeatletter
\patchcmd{\env@cases}{1.2}{.8}{}{}
\makeatother

\begin{document}

\title[Differentiability of solutions]{Differentiability of  solutions for a degenerate fully nonlinear free transmission problem}

\author{Edgard A. Pimentel}
\address{CMUC, Department of Mathematics, University of Coimbra, 3000-143 Coimbra, Portugal.}
\email{edgard.pimentel@mat.uc.pt}

\author{David Stolnicki}
\address{University of Warsaw, Krakowskie Przedmiescie 26/28, 00-927 Warszawa.}
\email{d.stolnicki-gonzalez@uw.edu.pl}

\date{\today} 

\keywords{Degenerate fully nonlinear free transmission problems; differentiability of solutions; Dini continuity.}

\subjclass[2020]{35B65.}

\maketitle

\begin{abstract}

\noindent We study a fully nonlinear free transmission problem in the presence of general degeneracy laws. Under minimal conditions on the degeneracy of the model, we establish the differentiability of viscosity solutions.

\end{abstract}

\vspace{.1in}

\section{Introduction}\label{sec_intro}

We consider viscosity solutions of
\begin{equation}\label{eq_main}
	\begin{split}
        \sigma_1(|Du|)F(D^2u)=f&\hspace{.2in}\mbox{in}\hspace{.2in}\Omega\cap\left\lbrace u>0 \right\rbrace\\
        \sigma_2(|Du|)F(D^2u)=f&\hspace{.2in}\mbox{in}\hspace{.2in}\Omega\cap\left\lbrace u<0 \right\rbrace,
	\end{split}
\end{equation}
where $F:S(d)\to\mathbb{R}$ is a uniformly elliptic operator, $f\in L^\infty(\Omega)\cap C(\overline{\Omega})$, and $\sigma_i:\mathbb{R}_+\to\mathbb{R}_+$ are degeneracy rates. Here, $S(d)\sim\mathbb{R}^\frac{d(d+1)}{2}$ denotes the space of symmetric matrices of order $d$, whereas $\Omega\subset\mathbb{R}^d$ is open, bounded and connected.

Concerning the degeneracy rates $\sigma_1$ and $\sigma_2$, we require these functions to be monotone. We also suppose $\sigma_1$ and $\sigma_2$ to have inverses $\sigma_1^{-1}$ and $\sigma_2^{-1}$, which are themselves Dini moduli of continuity. Under such a condition, we prove the differentiability of the solutions to \eqref{eq_main}.

The model in \eqref{eq_main} amounts to a degenerate fully nonlinear free boundary problem. In particular, \eqref{eq_main} describes a diffusion process degenerating as a modulus of continuity of the gradient. Meanwhile, the degeneracy depends on the sign of the solution. 

Discontinuous diffusions have been studied in the literature since the work of Mauro Picone, circa 1950; see \cite{Picone1954}. The problem formulated in \cite{Picone1954} finds its roots in the realm of elasticity theory. We refer the reader to \cite{Borsuk2010}. After the establishment of a soundly based theory of the existence of solutions, the analysis of regularity properties took place. For example, we mention \cite{li1,li2}.

More recently, a corpus of finer regularity results appeared in the context of transmission problems. Here, the authors examine the regularity of the solutions under \emph{minimal} regularity conditions on the transmission interface; see \cite{CafCarSti2021, SorSti2023}, 

In recent years, the analysis of transmission problems started to account for diffusions with discontinuities across \emph{solution-dependent} interfaces. The analysis of a variational, uniformly elliptic free transmission problem appeared recently in \cite{AmaTei2015}. See \cite{Shag1} for a variational free transmission problem modelled after the $p$-Laplace operator. 
 
Free transmission problems driven by uniformly elliptic fully nonlinear operators have been addressed in \cite{PimSwi2022}, \cite{PimSan2020}, and \cite{CarKri2024}. Degenerate fully nonlinear free transmission problems are the subject of \cite{HuaPimRamSwi2021}. The authors consider
\begin{equation}\label{eq_mainswi}
	\begin{split}
		\left|Du\right|^{\theta_1}F(D^2u)&=f\hspace{.2in}\mbox{in}\hspace{.2in}\Omega\cap\left\lbrace u>0 \right\rbrace\\
		\left|Du\right|^{\theta_2}F(D^2u)&=f\hspace{.2in}\mbox{in}\hspace{.2in}\Omega\cap\left\lbrace u<0 \right\rbrace,
	\end{split}
\end{equation}
where $F$ is a uniformly elliptic operator and $0<\theta_1<\theta_2$ are fixed constants. They prove the existence of $L^p$-viscosity solutions and an optimal regularity theory in H\"older spaces. The estimates in \cite{HuaPimRamSwi2021} depend explicitly on $\theta_1$ and $\theta_2$. See \cite{Cristiana, DavidJesus} for the analysis of interesting variants of the model in  \eqref{eq_mainswi}.

The model in \eqref{eq_mainswi} is motivated by the study of fully nonlinear equations degenerating as a power of the gradient. The workhorse of the theory is the equation
\begin{equation}\label{eq_horse}
	|Du|^\theta F(D^2u)=f\hspace{.2in}\mbox{in}\hspace{.2in}\Omega,
\end{equation}
where $\theta\in(-1,\infty)$ is a fixed constant, $F$ is a uniformly elliptic operator, and $f\in L^\infty(\Omega)\cap C(\overline\Omega)$. It is well-known that solutions to \eqref{eq_horse} are locally $C^{1,\alpha}$-regular; see \cite{Imbert-Silvestre2013,Birindelli-Demengel2014,Araujo-Ricarte-Teixeira2015}. 

A generalisation of the degeneracy law $p\mapsto |p|^\theta$ appeared in \cite{AndPelPimTei2022}. In that paper, the authors propose an equation of the form
\begin{equation}\label{eq_clea}
	\sigma\left(\left|Du\right|\right) F(D^2u)=f\hspace{.2in}\mbox{in}\hspace{.2in}\Omega,
\end{equation}
where $\sigma:\mathbb{R}_+\to\mathbb{R}_+$ is a modulus of continuity whose inverse $\sigma^{-1}$ is itself a Dini-modulus of continuity. Working under such a condition on $\sigma$, the authors prove that solutions to \eqref{eq_clea} are locally of class $C^1$.

The connection of \eqref{eq_horse} and \eqref{eq_mainswi} is subtle. The seminal contribution connecting these models appeared in \cite{BroRamPimTei}. In that paper, the authors examine a \emph{state-dependent} degeneracy rate $\theta:\Omega\to\mathbb{R}$, giving rise to
\begin{equation}\label{eq_bprt}
	|Du|^{\theta(x)}F(D^2u)=f\hspace{.2in}\mbox{in}\hspace{.2in}\Omega.
\end{equation}
In \cite{BroRamPimTei}, the authors establish the optimal regularity of the solutions in $C^{1,\alpha}$ spaces. A delicate argument yields optimal regularity estimates \emph{not depending} on the modulus of continuity of $\theta(\cdot)$. This fine analysis frames degenerate free boundary problems as a limit case of a model driven by continuous variable exponents. Compare \eqref{eq_mainswi} with \eqref{eq_bprt}, where $\theta(x)$ is an approximation of the exponential rule switching from $\theta_1$ to $\theta_2$. 

Motivated by the ideas launched in \cite{BroRamPimTei}, we extend the approach in \cite{HuaPimRamSwi2021} to \eqref{eq_main}. Without loss of generality, we set $\Omega\equiv B_1$ and state our main result as follows.

\begin{theorem}[Interior $C^1$-regularity estimates]\label{theo_regularity}
 Let $u\in C(B_1)$ be a viscosity solution to \eqref{eq_main}. Suppose $F$ is a $(\lambda,\Lambda)$-uniformly elliptic operator and $f\in L^\infty(\Omega)\cap C(\overline\Omega)$. Suppose also $\sigma_i$ are monotone increasing. Suppose further that $\sigma_i^{-1}$ is a Dini-continuous modulus of continuity. Then $u \in C^1_{{\rm loc}}(B_1)$. Moreover, there exists a modulus of continuity $\omega: \mathbb{R}_0^+ \to \mathbb{R}_0^+$ depending only upon the dimension, ellipticity constants, $\sigma_i$, $\|u\|_{L^\infty(B_1)}$, and $\|f\|_{L^\infty(B_1)}$ such that
\[ |Du(x) - Du(y)| \leq \omega(|x - y|), \]
for every $x, y \in B_{1/4}$.
\end{theorem}

The proof of Theorem \ref{theo_regularity} relies on several ingredients. First, we notice that solutions to \eqref{eq_main} satisfy two viscosity inequalities in the entire domain $B_1$. Then we resort to standard approximation methods to construct hyperplanes locally comparable with the solution, except for a prescribed error. Finally, the summability due to the Dini continuity of $\sigma_i^{-1}$ ensures the convergence of such hyperplanes and concludes the proof. A by-product of our argument is an explicit modulus of continuity for $Du$.

The remainder of this paper is organised as follows. Section \ref{sec_prelim} details our main assumptions and gathers definitions and former results used in the paper. The proof of Theorem \ref{theo_regularity} is the subject of Section \ref{sec_regularity}.

\section{Preliminaries}\label{sec_prelim}

In this section, we gather preliminary results and notions used throughout the paper. We start by stating our main assumptions.
\begin{Assumption}[Uniform ellipticity]\label{assump_Felliptic}
Fix $0<\lambda\leq\Lambda$. We suppose that the operator $F:S(d)\to\mathbb{R}$ is $(\lambda,\Lambda)$-elliptic. That is, for every $M,\,N\in S(d)$ we have
\[
    \lambda\left\|N\right\|\leq F(M+N)-F(M)\leq \Lambda\left\|N\right\|,
\]
provided $N\geq 0$.
\end{Assumption}

We continue with an assumption on the degeneracy laws.

\begin{Assumption}[Degeneracy rates]\label{assump_sigmabasic}
The degeneracy rates $\sigma_1,\,\sigma_2:\mathbb{R}_+\to\mathbb{R}_+$ are continuous, monotone increasing, with
\[
	\lim_{t\to0}\sigma_i(t)=0,
\]
for $i=1,2$. In addition, they have inverses $\sigma_1^{-1}$ and $\sigma_2^{-1}$ that are Dini-continuous moduli of continuity.
\end{Assumption}

A typical example of degeneracy laws satisfying Assumptions \ref{assump_sigmabasic} is
\[
    \sigma_1(t)\coloneqq t^{p_1}\hspace{.3in}\mbox{and}\hspace{.3in}\sigma_2(t)\coloneqq t^{p_2},
\]
Also, $\sigma_i^{-1}(t)=t^{1/p_i}$, which is H\"older continuous with exponent $1/p_i$. In what follows, we gather auxiliary results used in the paper. We continue with the notion of viscosity solutions.

\begin{definition}[Viscosity solution]\label{def_cvisc}
Let $G:S(d)\times\mathbb{R}^d\times\mathbb{R}\times\Omega\to\mathbb{R}$ be a degenerate elliptic operator. We say that $u\in{\rm USC}(\Omega)$ {\rm [resp. }$u\in{\rm LSC}(\Omega)${\rm ]} is a viscosity sub-solution {\rm [resp. }super-solution{\rm ]} to 
\begin{equation}\label{eq_visc}
G(D^2u,Du,u,x)=0\hspace{.3in}\mbox{in}\hspace{.3in}\Omega
\end{equation}
if, whenever $\varphi\in C^2(\Omega)$ and $u-\varphi$ attains a maximum {\rm [resp. }minimum{\rm ]} at $x_0\in\Omega$, we have
\[
    \begin{split}
        G(&D^2\varphi(x_0),D\varphi(x_0),u(x_0),x_0)\leq 0\\
            {\rm [resp. \;}&G(D^2\varphi(x_0),D\varphi(x_0),u(x_0),x_0)\geq 0{\rm ]}.
    \end{split}
\]
If $u\in C(\overline{\Omega})$ is both a viscosity sub-solution and a $C$-viscosity super-solution to \eqref{eq_visc}, we say that $u$ is a viscosity solution to \eqref{eq_visc}.
\end{definition}

In case a viscosity solution $u\in C(\overline\Omega)$ is such that $\left\|u\right\|_{L^\infty(\Omega)}\leq 1$, we refer to $u$ as a \emph{normalized} viscosity solution. We continue with the definition of the extremal Pucci operators. Indeed, for $0<\lambda\leq\Lambda$, define $\mathcal{A}_{\lambda,\Lambda}\subset S(d)$ as
\[
	\mathcal{A}_{\lambda,\Lambda}\coloneqq \left\lbrace A\in S(d)\;|\;\lambda|\xi|^2\leq A\xi\cdot\xi\leq\Lambda|\xi|^2\hspace{.1in}\mbox{for every}\hspace{.1in}\xi\in\mathbb{R}^d\right\rbrace.
\]
The extremal operators are defined as follows. 

\begin{definition}[Extremal Pucci operators]\label{def_pucci}
Let $0<\lambda\leq\Lambda$ be fixed, though arbitrary. The extremal Pucci operator $\mathcal{M}^-_{\lambda,\Lambda}:S(d)\to\mathbb{R}$ is given by
\[
	\mathcal{M}^-_{\lambda,\Lambda}(M)\coloneqq \inf_{A\in\mathcal{A}_{\lambda,\Lambda}}\,{\rm Tr}(AM).
\]
Also, define $\mathcal{M}^+_{\lambda,\Lambda}(M)\coloneqq -\mathcal{M}^-_{\lambda,\Lambda}(-M)$.
\end{definition}

Among other things, the extremal operators are useful in defining uniform ellipticity. Indeed, an operator $F:S(d)\to\mathbb{R}$ satisfies Assumption \ref{assump_Felliptic} if, for any $M,\,N\in S(d)$, we have
\[
    \mathcal{M}_{\lambda,\Lambda}^-(M-N)\leq F(M)-F(N)\leq \mathcal{M}_{\lambda,\Lambda}^+(M-N).
\]

We continue by recalling auxiliary results used to produce compactness in our context.

\begin{proposition}[Maximum principle]\label{prop_jil}
    Let $H,G\in C(S(d)\times\mathbb{R}^d\times\Omega)$ be degenerate elliptic operators. Let $u\in{\rm USC}(\Omega)$ be a $C$-viscosity sub-solution to $G(D^2u,Du,x)=0$ in $\Omega$ and $v\in{\rm LSC}(\Omega)$ be a $C$-viscosity super-solution to $H(D^2v,Dv,x)=0$ in $\Omega$. Define $w:\Omega\times\Omega\to\mathbb{R}$ as
    \[
        w(x,y)\coloneqq u(x)-v(y).
    \]
    Let $\Psi\in C^2(\Omega\times\Omega)$ and suppose $(\overline x,\overline y)\in\Omega\times\Omega$ is a local maximum point for $w-\Psi$. For every $\varepsilon>0$ there exist $X,\,Y\in S(d)$ such that 
    \[
        G(X,D_x\Psi(\overline x,\overline y),\overline x)\leq 0\leq H(Y,D_y\Psi(\overline x,\overline y),\overline y). 
    \]
    In addition,
\[
-\left(\frac{1}{\varepsilon}+\|D^2\Psi(\overline x,\overline y)\|\right)I\,\leq\,
\left(
\begin{array}{ccc}
X   & 0 \\
0  &-Y 
\end{array}
\right)
\leq 
D^2\Psi(\overline x,\overline y)+\varepsilon \left[D^2\Psi(\overline x,\overline y)\right]^2.
\]
\end{proposition}
For a proof of Proposition \ref{prop_jil}, we refer the reader to \cite[Theorem 3.2]{CraIshLio1992}. We continue with a variant of a regularity result for equations holding \emph{only where the gradient is large}. See \cite{ImbSil2016}; see also \cite{Moo2015}.

\begin{proposition}[H\"older regularity]\label{prop_holder1}
    Fix $\gamma>0$. Let $u\in C(\overline{\Omega})$ be a normalized viscosity solution to 
    \[
        \begin{split}
            \mathcal{M}_{\lambda,\Lambda}^-(D^2u)\leq C_0&\hspace{.2in}\mbox{in}\hspace{.2in}\Omega\cap\left\lbrace|Du|\geq \gamma \right\rbrace\\
            \mathcal{M}_{\lambda,\Lambda}^+(D^2u)\geq -C_0&\hspace{.2in}\mbox{in}\hspace{.2in}\Omega\cap\left\lbrace|Du|\geq \gamma \right\rbrace.
        \end{split}
    \]
    There exists $\beta\in(0,1)$ such that $u\in C^\beta_{\rm loc}(\Omega)$. Moreover, for $\Omega'\Subset\Omega$ there exists $C>0$ such that
    \[
        \left\|u\right\|_{C^\beta(\Omega')}\leq CC_0.
    \]
Finally, $\beta=\beta(\lambda,\Lambda,d)$ and $C=C(\lambda,\Lambda,d,\gamma,{\rm diam}(\Omega),{\rm dist}(\Omega',\partial\Omega))$.

\end{proposition}

We close this section with a proposition relating the sequence spaces $\ell_1(\mathbb{R}^d)$ and $c_0(\mathbb{R}^d)$. See \cite[Lemma 1]{AndPelPimTei2022}.
\begin{proposition}\label{prop_modulator}
	Let $(a_j)_{j \in \mathbb{N}} \in \ell^1$ and take $\varepsilon,\delta > 0$, arbitrary. There exists a sequence $(c_j)_{j \in \mathbb{N}} \in c_0$, with $\max_{j \in \mathbb{N}} |c_j| \leq \varepsilon^{-1}$, such that 
	\[
	\left( \frac{a_j}{c_j} \right)_{j \in \mathbb{N}}\in \ell^1
	\]
	and 
	\[
	\varepsilon\left(1-\frac{\delta}{2}\right)\left\| (a_j) \right\|_{\ell_1} \leq \left\| \left( \frac{a_j}{c_j} \right)\ \right\|_{\ell_1} \leq \varepsilon(1 + \delta) \left\| (a_j) \right\|_{\ell_1}.
	\]
\end{proposition} 

\section{Local $C^1$-regularity estimates}\label{sec_regularity}

In this section, we detail the proof of Theorem \ref{theo_regularity}. Our strategy is based on a technique introduced in \cite{PimSwi2022,HuaPimRamSwi2021}, relating \eqref{eq_main} with a pair of viscosity inequalities holding in the \emph{entire} domain. Indeed, if $u\in C(\overline\Omega)$ is a viscosity solution to \eqref{eq_main}, it solves
\begin{equation*}\label{eq_viscmin}
    \min\left(\sigma_1(|Du|)F(D^2u),\sigma_2(|Du|)F(D^2u)\right)\leq C_0\hspace{.2in}\mbox{in}\hspace{.2in}\Omega
\end{equation*}
and
\begin{equation*}\label{eq_viscmax}
    \max\left(\sigma_1(|Du|)F(D^2u),\sigma_2(|Du|)F(D^2u)\right)\geq -C_0\hspace{.2in}\mbox{in}\hspace{.2in}\Omega.
\end{equation*}
In the sequel, we prove that viscosity solutions to \eqref{eq_main} are locally H\"older continuous, with estimates. For simplicity, and without loss of generality, we set $\Omega\equiv B_1$. The next result is a direct consequence of Proposition \ref{prop_holder1}.

\begin{proposition}[H\"older continuity]\label{prop_holderq1}
    Let $u\in C(B_1)$ be a viscosity solution to
    \begin{equation}\label{eq_minq}
        \min\left(\sigma_1(|q+Du|)F(D^2u),\sigma_2(|q+Du|)F(D^2u)\right)\leq C_0\hspace{.2in}\mbox{in}\hspace{.2in}B_1
    \end{equation}
    and
    \begin{equation}\label{eq_maxq}
        \max\left(\sigma_1(|q+Du|)F(D^2u),\sigma_2(|q+Du|)F(D^2u)\right)\geq -C_0\hspace{.2in}\mbox{in}\hspace{.2in}B_1,
    \end{equation}
    where $q\in\mathbb{R}^d$ is fixed, though arbitrary. Suppose Assumptions \ref{assump_Felliptic} and \ref{assump_sigmabasic} are in force. Suppose further $|q|\leq A_0$, for some constant $A_0>1$. Then there exists $\beta\in(0,1)$ such that $u\in C^\beta_{\rm loc}(B_1)$. In addition, for every $\tau\in(0,1)$ there exists $C_\tau>0$ for which
    \[
        \left\|u\right\|_{C^\beta(B_\tau)}\leq C_\tau.
    \]
    The constant $C_\tau$ depends on $\lambda$, $\Lambda$, the dimension $d$, $C_0$ and $\sigma_i(A_0)$, for $i=1,\,2$.
\end{proposition}
\begin{proof}
Let $p$ be such that $|p|>2A_0$. Hence $|q+p|>A_0$. This inequality builds upon Assumption \ref{assump_sigmabasic} to ensure 
\[
    \overline\sigma\coloneqq \min\left(\sigma_1(A_0),\sigma_2(A_0)\right)\leq\min\left(\sigma_1(|q+p|),\sigma_2(|q+p|)\right).
\]
Hence, $u$ solves
\[
    \mathcal{M}_{\lambda,\Lambda}^-(D^2u)\leq \frac{C_0}{\overline\sigma}\hspace{.2in}\mbox{in}\hspace{.2in}B_1\cap\left\lbrace|p|>2A_0 \right\rbrace
\]
and
\[
    \mathcal{M}_{\lambda,\Lambda}^+(D^2u)\geq -\frac{C_0}{\overline\sigma}\hspace{.2in}\mbox{in}\hspace{.2in}B_1\cap\left\lbrace|p|>2A_0 \right\rbrace.
\]
By setting $\gamma\coloneqq 2A_0$ in Proposition \ref{prop_holder1} the result follows.
\end{proof}

We continue with an application of the maximum principle to produce H\"older continuity for a solution to \eqref{eq_minq}-\eqref{eq_maxq} in case $|q|\geq A_0$. Our argument follows along the same lines as in \cite[Proposition 5]{HuaPimRamSwi2021}.

\begin{proposition}[H\"older continuity]\label{prop_holdercont1}
Let $u\in C(B_1)$ be a viscosity solution to \eqref{eq_minq}-\eqref{eq_maxq}. Suppose Assumptions \ref{assump_Felliptic} and \ref{assump_sigmabasic} are in force. Then $u\in C^\beta_{\rm loc}(B_1)$, for some universal constant $\beta\in(0,1)$. In addition, for every $\rho\in(0,1)$ there exists $C>0$ such that
\[
    \left\|u\right\|_{C^\beta(B_\rho)}\leq C.
\]
\end{proposition}
\begin{proof}
We argue as in the proof of \cite[Proposition 5]{HuaPimRamSwi2021}; in the sequel, we omit most of the details, stressing the main differences with respect to the argument in that paper. We split the proof into four steps.

\bigskip

\noindent{\bf Step 1 - }Fix $0<r<(1-\tau)/2$ and consider the modulus of continuity $\omega(t)\coloneqq t-t^2/2$. Define the quantity
\[
    L\coloneqq \sup_{x,y\in B_r(x_0)}\left(u(x)-u(y)-L_1\omega(|x-y|)-L_2\left(|x-x_0|^2+|y-x_0|^2\right)\right).
\]
As usual, our goal is to choose $L_1$ and $L_2$ such that $L\leq 0$ for every $x_0\in B_\tau$. Suppose such a choice is not possible. It is tantamount to say there exists $x_0$ such that $L>0$ regardless of the choice of $L_1$ and $L_2$. Arguing as in \cite[Proposition 5]{HuaPimRamSwi2021}, we find points $(X,p_x,x),\,(Y,p_y,y)\in S(d)\times\mathbb{R}^d\times B_r(x_0)$, and a constant $\iota>0$, such that
\begin{equation}\label{eq_contradiction1}
\mathcal{M}_{\lambda,\Lambda}^-(X-Y)\geq 4\lambda L_1-(\lambda+(d-1)\Lambda)(4L_2+2\iota),
\end{equation}
\begin{equation}\label{eq_contradiction2}
    \min\left(\sigma_1(|q+p_x|)F(X),\sigma_2(|q+p_x|)F(X)\right)\leq C_0,
\end{equation}
\begin{equation}\label{eq_contradiction3}
    \max\left(\sigma_1(|q+p_y|)F(Y),\sigma_2(|q+p_y|)F(Y)\right)\leq C_0,
\end{equation}
and
\begin{equation}\label{eq_contradiction4}
    F(X)\geq F(Y)+\mathcal{M}_{\lambda,\Lambda}^-(X-Y).
\end{equation}

Combining \eqref{eq_contradiction1}-\eqref{eq_contradiction4}, we get
\begin{equation}\label{eq_contradiction5}
    4\lambda L_1\leq \left(\lambda+(d-1)\Lambda\right)\left(4L_2+2\iota\right)+C_0\left(\frac{1}{\sigma_i(|q+p_j|)}+\frac{1}{\sigma_k(|q+p_\ell|)}\right),
\end{equation}
where $i,k\in\{1,2\}$, and $j,\ell\in\{x,y\}$.

\bigskip

\noindent{\bf Step 2 - }Because 
\[
|p_x|,|p_y|\leq L_1(1+|x-y|)+2L_2,
\]
where $L_2\coloneqq (4\sqrt{2}/r)^2$, we conclude there exists $a>0$ such that 
\[
|p_x|,|p_y|\leq aL_1.
\]

Set $A_0=10aL_1$ and suppose $|q|>A_0$. For those choices, it is clear that $q\neq p_x$ and $q\neq p_y$. Moreover,
\[
    \left|q+p_j\right|\geq A_0-\frac{A_0}{10}\geq \frac{9}{10}A_0,
\]
for $j\in\{x,y\}$. Therefore,
\[
    \frac{1}{\sigma_i(|q+p_j|)}\leq\frac{1}{\sigma_i(|q+p_j|)}=\frac{1}{\sigma_i(9aL_1)},
\]
for $i\in \{1,2\}$. Hence, \eqref{eq_contradiction5} becomes
\[
    4\lambda L_1\leq \left(\lambda+(d-1)\Lambda\right)\left(4L_2+2\iota\right)+C(L_1)C_0,
\]
where the constant $C(L_1)$ is monotone decreasing in $L_1$. By taking $L_1>0$ large enough, one gets a contradiction and proves that solutions to \eqref{eq_minq}-\eqref{eq_maxq} are locally Lipschitz continuous if $|q|\geq A_0$. 

\bigskip
\noindent{\bf Step 3 - }In case $|q|\leq A_0$, Proposition \ref{prop_holderq1} ensures the H\"older continuity of $u$. This fact builds upon the conclusion in Step 2 to complete the proof.
\end{proof}

The compactness stemming from the former result unlocks an approximation lemma, instrumental in our analysis. This is the content of the next proposition.

\begin{proposition}[Approximation Lemma]\label{prop_approximation}
Let $u\in C(B_1)$ be a viscosity solution to \eqref{eq_minq}-\eqref{eq_maxq}. Suppose Assumptions \ref{assump_Felliptic} and \ref{assump_sigmabasic} are in force. Let $\alpha_0\in(0,1)$ be the exponent for the Krylov-Safonov regularity theory available for $F=0$. For every $\delta>0$ there exists $\varepsilon>0$ such that, if $C_0 \leq \varepsilon$ then one can find $h \in C^{1,\alpha_0}_{\rm loc}(B_{9/10})$ satisfying 
\[
    \left\|u - h \right\|_{L^\infty(B_{9/10})} \leq \delta. 
\]
In addition, there exists $C>0$ for which
\begin{equation*}
\left\| h \right\|_{C^{1,\alpha_0}(B_{8/9})} \leq C.
\end{equation*}
Finally, the constant $C = C(d,\lambda,\Lambda)>0 $ is independent of $q$.
\end{proposition}

\begin{proof}
The proof resorts to a contradiction argument. For ease of presentation, we split it into six steps.

\medskip

\noindent{\bf Step 1 - }Suppose the result does not hold. In this case, there are sequences $(\sigma_1^n)_{n\in\mathbb{N}}$, $(\sigma_2^n)_{n\in\mathbb{N}}$, $(q_n)_{n\in\mathbb{N}}$, $(u_n)_{n\in\mathbb{N}}$, $(F_n)_{n\in\mathbb{N}}$, $(f_n)_{n\in\mathbb{N}}$ and $\delta_0>0$ such that:
\begin{enumerate}
    \item The operator $F_n$ satisfies Assumption \ref{assump_Felliptic}, for every $n\in \mathbb{N}$;
    \item The functions $\sigma_1^n$ and $\sigma_2^n$ are such that $\sigma_i^n(0)=0$, $\sigma_i^n(1)\geq 1$ and, if $\sigma_i^n(a_n)\to0$, then $a_n\to 0$;
    \item The function $f_n\in L^\infty(B_1)\cap C(\overline{B_1})$ is such that 
    \[
        \left\|f_n\right\|_{L^\infty(B_1)}=:C_n\to0\qquad\mbox{as}\quad n\to\infty;
    \]
    \item The following inequalities hold in the viscosity sense:
    \[
        \min\left(\sigma_1^n\left(\left|Du_n+q_n\right|\right)F_n(D^2u_n),\sigma_2^n\left(\left|Du_n+q_n\right|\right)F_n(D^2u_n)\right)\leq C_n
    \]
    and
    \[
        \max\left(\sigma_1^n\left(\left|Du_n+q_n\right|\right)F_n(D^2u_n),\sigma_2^n\left(\left|Du_n+q_n\right|\right)F_n(D^2u_n)\right)\geq -C_n,
    \]
    in the unit ball $B_1$;
    \item We have 
    \[
        \sup_{x\in B_{7/8}}\left|u_n(x)-h(x)\right|>\delta_0
    \]
    for every $n\in\mathbb{N}$, and every $h\in C^{1,\alpha_0}_{\rm loc}(B_{8/9})$.
\end{enumerate}
To produce a contradiction, we use the compactness available for the sequence $(u_n)_{n\in\mathbb{N}}$ and the uniform ellipticity of $F_n$. This is the subject of the next steps.

\medskip

\noindent{\bf Step 2 - } Because of Proposition \ref{prop_holdercont1}, there exists a subsequence, still denoted with $(u_n)_{n\in\mathbb{N}}$, converging uniformly to some $u_\infty \in C^\beta_{\text{\rm loc}}({B_1})$. Also, Assumption \ref{assump_Felliptic} implies that $(F_n)_{n \in \mathbb{N}}$ is a sequence of uniformly Lipschitz-continuous operators. As a consequence, there exists $F_\infty$ satisfying Assumption \ref{assump_Felliptic} such that $F_n$ converges to $F_\infty$, locally uniformly (through some subsequence if necessary). Our goal is to prove that $u_\infty$ is a viscosity solution to $F_\infty(D^2w) = 0$ in $B_1$. We only show a sub-solution property, as its super-solution counterpart is entirely analogous. Consider the paraboloid $p(x)$ defined as
\[
    p(x)\coloneqq u_\infty(y)+{\rm b}\cdot (x-y)+\frac{1}{2}(x-y)^TM(x-y).
\]
Suppose $p$ touches $u_\infty$ from above in a vicinity of $y\in B_1$. Consider also the sequence $(x_n)_{n\in\mathbb{N}}$ such that $p$ touches $u_n$ from above at $x_n$ and $x_n\to y$ as $n\to \infty$. We then have
\begin{equation}\label{eq_bennygantz}
	\min\left(\sigma_1^n\left(\left|{\rm b}+q_n\right|\right)F_n(M),\sigma_2^n\left(\left|{\rm b}+q_n\right|\right)F_n(M)\right)\leq C_n.
\end{equation}
The proof is complete if we verify $F_\infty(M)\leq 0$. To reach this conclusion we split the remainder of our argument into several cases, depending on the behaviour of ${\rm b}+q_n$.

\medskip

\noindent{\bf Step 3 - }Suppose the sequence $(q_n)_{n \in \mathbb{N}}$ does not admit a convergent subsequence. That is, $|q_n| \to \infty$ as $n \to \infty$. Then there exists $N\in\mathbb{N}$ such that 
\[
	\left|{\rm b}+q_n\right|\geq1
\]
for $n>N$. In this case, \eqref{eq_bennygantz} implies
\[
	F_n(M)\leq\frac{C_n}{\sigma_1^n\left(\left|{\rm b}+q_n\right|\right)}\leq C_n\qquad\mbox{or}\qquad F_n(M)\leq\frac{C_n}{\sigma_2^n\left(\left|{\rm b}+q_n\right|\right)}\leq C_n.
\]
In any case, we get $F_n(M)\leq C_n$. By taking the limit $n\to\infty$, one recovers $F_\infty(M)\leq 0$ and completes the proof in this case. It remains to examine the case where $(q_n)_{n\in\mathbb{N}}$ is bounded.

\medskip

\noindent{\bf Step 4 - } If the sequence $(q_n)_{n \in \mathbb{N}}$ is bounded, at least through a subsequence it converges to some $q_\infty\in\mathbb{R}^d$. Suppose 
\[
	\left|{\rm b}+q_\infty\right|=:\tau>0.
\]
Then one can find $N\in\mathbb{N}$ such that 
\[
	\left|{\rm b}+q_n\right|>\frac{\tau}{2},
\]
provided $n>N$. Hence, $\sigma_i^n\left(\left|{\rm b}+q_n\right|\right)>\sigma_i^n(\tau/2)$ for every $n>N$ and the previous argument easily adjusts to yield
\[
	F_n(M)\leq\frac{C_n}{\sigma_1^n\left(\tau/2\right)}\qquad\mbox{or}\qquad F_n(M)\leq\frac{C_n}{\sigma_2^n\left(\tau/2\right)}.
\]
In any case, $F_n(M)\leq \overline{C_n}$, with $\overline{C_n}\to0$ as $n\to \infty$. Once again we recover $F_\infty(M)\leq 0$. It remains to study the case $\left|{\rm b}+q_\infty\right|=0$.

\medskip

\noindent{\bf Step 5 - }Without loss of generality we suppose ${\rm b}=0$ and $y=0$. Also, suppose $M$ has $k\in\left\lbrace 1,\ldots,d\right\rbrace$ strictly positive eigenvalues. Indeed, were all the eigenvalues of $M$ non-positive, ellipticity would ensure $F_n(M)\leq 0$ for every $n\in\mathbb{N}$, leading immediately to the desired conclusion.

For $i=1,\ldots,k$, denote with $e_i$ the eigenvector associated with the $i$-th strictly positive eigenvalue of $M$. Define $E$ as the subspace of $\mathbb{R}^d$ spanned by $e_1,\ldots, e_k$ and write $\mathbb{R}^d=:E\oplus G$. Finally, consider the test function
\[
\varphi(x)\coloneqq \kappa\sup_{e\in\mathbb{S}^{d-1}}\left\langle P_Ex,e\right\rangle+\frac{1}{2}x^TMx,
\]
where $0<\kappa\ll1$ is a fixed constant, $P_E$ is the orthogonal projection into $E$ and $\mathbb{S}^{d-1}$ is the unit sphere of dimension $(d-1)$. We notice that usual stability results ensure that $\varphi$ touches $u_n$ from above at some point $x_n^\kappa$, with $x_n^\kappa \to 0$ as $n\to \infty$.
Note that $\varphi$ is ${C}^2$ outside $G$. If $x_n^\kappa \in G $, we modify the function $\varphi$ to consider
\[
\varphi_e(x)\coloneqq \kappa\left\langle P_Ex,e\right\rangle+\frac{1}{2}x^TMx.
\]
Hence,
\begin{equation}\label{eq_arrow}
\min\left(\sigma_1^n\left(\left|\kappa e+Mx_n+q_n\right|\right)F_n(M),\sigma_2^n\left(\left|\kappa e+Mx_n+q_n\right|\right)F_n(M)\right)\leq C_n.
\end{equation}
Choosing,
\[
e\coloneqq \frac{Mx_n}{\left|Mx_n\right|}
\]
and noticing that
\[
\frac{\kappa}{2}\leq \kappa-\left|q_n\right|\leq \left|\kappa e+Mx_n+q_n\right|,
\]
for large enough $n\gg1$, inequality \eqref{eq_arrow} yields
\begin{equation}\label{eq_danhagari}
F_n(M)\leq\frac{C_n}{\sigma_1^n\left(\kappa/2\right)}\qquad\mbox{or}\qquad F_n(M)\leq\frac{C_n}{\sigma_2^n\left(\kappa/2\right)},
\end{equation}
for every $n\in \mathbb{N}$.

Now, suppose $P_Ex_n\neq 0$. Arguing as before, we get that either
\[
\sigma_1^n\left(\left|Mx_n+\kappa\frac{P_Ex_n}{|P_Ex_n|}+q_n\right|\right)F_n\left(M+\kappa\left(Id+\frac{P_Ex_n}{|P_Ex_n|}\otimes\frac{P_Ex_n}{|P_Ex_n|}\right)\right)\leq C_n
\]
or
\[
\sigma_2^n\left(\left|Mx_n+\kappa\frac{P_Ex_n}{|P_Ex_n|}+q_n\right|\right)F_n\left(M+\kappa\left(Id+\frac{P_Ex_n}{|P_Ex_n|}\otimes\frac{P_Ex_n}{|P_Ex_n|}\right)\right)\leq C_n.
\]
Notice that
\[
\frac{\kappa}{2}\leq \kappa-\left|q_n\right|\leq \left|Mx_n+\kappa\frac{P_Ex_n}{|P_Ex_n|}+q_n\right|;
\]
moreover,
\[
\kappa\left(Id+\frac{P_Ex_n}{|P_Ex_n|}\otimes\frac{P_Ex_n}{|P_Ex_n|}\right)\geq 0,
\]
in the sense of matrices. Ellipticity builds upon the monotonicity of $\sigma_i^n$ to produce \eqref{eq_danhagari} also when $P_ex_n\neq 0$. By taking the limit $n\to\infty$ in \eqref{eq_danhagari}, we get $F_\infty(M)\leq 0$. We have established that $u_\infty$ is a viscosity solution to $F_\infty=0$ in $B_1$.

\medskip

\noindent{\bf Step 6 - }Because $F_\infty(D^2u_\infty)=0$ in $B_1$, standard regularity results imply that $u_\infty\in C^{1,\alpha_0}_{loc}(B_1)$ with estimates, where $\alpha_0\in(0,1)$ is the (universal) exponent stemming from the Krylov-Safonov theory. By taking $h\coloneqq u_\infty$, one finds a contradiction and concludes the proof.
\end{proof}

We use Proposition \ref{prop_approximation} to construct a sequence of approximating hyperplanes. The difference between the solution and such approximating hyperplanes behaves in a precise geometric fashion. Such a geometric control of the difference between the solution and a hyperplane is key to differentiability.

\subsection{Proof of Theorem \ref{theo_regularity}}\label{subsec_JR}

For $\alpha_0\in (0,1)$ as in Proposition \ref{prop_approximation}, 
choose $r\in(0,1)$ and $\mu_1>0$ such that 
\begin{equation}\label{eq_mu13}
    2Cr^{1+\alpha_0}=\mu_1r,
\end{equation}
where $C>0$ is the universal constant in Proposition \ref{prop_approximation}. 

We proceed by defining the constant $\theta = \frac{r}{\mu_1}$ and considering the sequence $ (a_k)_{k \in \mathbb{N}}$ defined as
\begin{equation*}
    a_k \coloneqq  \max\left( \sigma_1^{-1}(\theta^k),\sigma_2^{-1}(\theta^k) \right).
\end{equation*}

Under Assumption \ref{assump_sigmabasic}, we conclude $(a_k)_{k \in \mathbb{N}} \in \ell^1$. Now, we resort to Proposition \ref{prop_modulator}. For $0 < \delta < \frac{1}{4}$, we set $0 < \varepsilon < 1$ as
\[
    \varepsilon\coloneqq \frac{1}{1+\delta}.
\]
For these choices, an application of Proposition \ref{prop_modulator} yields a sequence $(c_k)_{k \in \mathbb{N}}$ such that 
\begin{equation*}
\frac{7}{10}\sum\limits_{i=1}^{\infty} a_k \leq \sum\limits_{i=1}^{\infty} \frac{a_k}{c_k} \leq \sum\limits_{i=1}^{\infty} a_k<\infty.
\end{equation*}

Finally, we design two sequences of moduli of continuity $\left(\sigma_1^k(\cdot)\right)_{k\in\mathbb{N}}$ and $\left(\sigma_2^k(\cdot)\right)_{k\in\mathbb{N}}$ given by
\[
\begin{split}
\sigma_{i}^0(t)&\coloneqq  \sigma_i(t),\\
\sigma_{i}^1(t)&\coloneqq  \frac{\mu_1}{r} \sigma_i(\mu_1 t),\\
\sigma_{i}^2(t)&\coloneqq  \frac{\mu_1 \mu_2^*}{r^2} \sigma_i(\mu_1 \mu_2^* t),\\
&\vdots\\
\sigma_{i}^k(t)&\coloneqq  \frac{\mu_1\prod\limits_{j=2}^{k} \mu_j^*}{r^k} \sigma_i\left(\mu_1\prod\limits_{j=2}^{k} \mu_j^* t\right),
\end{split}
\]
with $\mu_1 > r$ as defined and $(\mu_k^*)_{k\in\mathbb{N}}$ determined as follows. Set
\[
	\mu_k^*\coloneqq \max\left(\mu_k^1,\mu^2_k\right).
\]
If 
\[
	\frac{\mu_{k-1}^*\prod\limits_{j=1}^{k-1} \mu^*_j}{r^k} \sigma_i\left(\mu_{k-1}^*\prod\limits_{j=1}^{k-1} \mu^*_j c_k\right) \geq 1,
\]
then $\mu_k^i = \mu_{k-1}^*$. Otherwise $\mu_k^i< 1$ is chosen to ensure
\[
	\frac{\mu_{k}^i\prod\limits_{j=1}^{k-1} \mu^*_j}{r^k} \sigma_i\left(\mu_k^i\prod\limits_{j=1}^{k-1} \mu_j c_k\right) = 1.
\]

As before, $\mu_k^*\coloneqq \max\left(\mu_k^1,\mu^2_k\right)$. Once these ingredients are available, we combine them with Proposition \ref{prop_approximation} to produce a sequence of affine functions whose difference with respect to $u$ grows in a controlled fashion.

\begin{proposition}\label{prop_discretehyperplane}
Let $u \in C(B_1)$ be a normalized viscosity solution to \eqref{eq_minq}-\eqref{eq_maxq}. Suppose Assumptions \ref{assump_Felliptic} and \ref{assump_sigmabasic} hold true. There exists $\varepsilon > 0$ such that, if $\left\|f\right\|_{L^{\infty}(B_1)} < \varepsilon$, one finds $a\in\mathbb{R}$ and ${\rm b}\in\mathbb{R}^d$ satisfying $|a| + \left\|b\right\| \leq C$, for some universal constant $C>0$. In addition,
\begin{equation*}
    \sup_{x\in B_r} \left|u(x) - (a+{\rm b}\cdot x)\right| \leq \mu_1 r,
\end{equation*}
where $\mu_1>0$ has been chosen in the algorithm described above.
\end{proposition}
\begin{proof}
We start by choosing the approximation parameter $\delta>0$ in Proposition \ref{prop_approximation}. Indeed, set
\[
    \delta\coloneqq \frac{\mu_1 r}{2}
\]
and let $\varepsilon>0$ be the corresponding smallness regime ensuring the existence of $h\in C^{1,\alpha_0}_{\rm loc}(B_{9/10})$, with $\left\|h\right\|_{C^{1,\alpha_0}(B_{8/9})}\leq C$, such that 
\begin{equation}\label{eq_destri1}
    \sup_{x\in B_{8/9}} |u(x) - h(x)| \leq \frac{\mu_1 r}{2},
\end{equation}
where the inequality follows from \eqref{eq_mu13}. The regularity available for $h$ implies
\begin{equation}\label{eq_destri2}
     \sup_{x\in B_r} |h(x) - h(0) - Dh(0) \cdot x| \leq Cr^{1+\alpha_0}.
\end{equation}

By combining \eqref{eq_destri1} and \eqref{eq_destri2}, one finds
\begin{equation*}
    \sup_{x\in B_{r}} \left|u(x) - (a + {\rm b} \cdot x)\right| \leq \mu_1 r,
\end{equation*}
and completes the proof.
\end{proof}

Now we extrapolate the findings in Proposition \ref{prop_discretehyperplane} to arbitrary small scales, in a discrete scheme.

\begin{proposition}[Oscillation control at discrete scales]\label{prop_osccontrol} 
    Let $u \in C(B_1)$ be a normalized viscosity solution of \eqref{eq_minq}-\eqref{eq_maxq}. Assume that Assumptions \ref{assump_Felliptic} and \ref{assump_sigmabasic} hold and that $\left\|f\right\|_{L^\infty(B_1)} \leq \varepsilon$, where $\varepsilon$ is the same as in Proposition \ref{prop_discretehyperplane}. Then for every $n \in \mathbb{N}$, there are affine functions $(\phi_n)_{n\in\mathbb{N}}$ of the form
    \begin{equation*}
        \phi_n(x) \coloneqq  a_n + {\rm b}_n \cdot x
    \end{equation*}
satisfying
    \begin{equation*}
        \sup_{x\in B_{r_n}} |u(x) - \phi_n(x)| \leq \left(\prod\limits_{j=1}^{n} \mu^*_j\right) r_n,
    \end{equation*}
    \begin{equation*}
        |a_{n+1} - a_n| \leq C \left(\prod\limits_{j=1}^{n} \mu_j^* \right) r_n,
    \end{equation*}
and
    \begin{equation*}
        |{\rm b}_{n+1} - {\rm b}_n| \leq C  \prod\limits_{j=1}^{n} \mu_j^*.
    \end{equation*}
\end{proposition}
\begin{proof}
    The proof follows from an induction argument. The case $n=1$ is covered by taking $\phi_1 \coloneqq  a+{\rm b}\cdot x$, where $a\in\mathbb{R}$ and ${\rm b}\in\mathbb{R}^d$ are as in Proposition \ref{prop_discretehyperplane}. It accounts for the base case.
    
    Suppose the case $n=k-1$ has already been established; it remains to verify the statement for $n=k$. To that end, consider the function
    \begin{equation}
        u_{k}(x) \coloneqq  \frac {u_{k-1}(r x) - \ell_{k-1}(rx)}{\mu^*_{k}\cdot r},
    \end{equation}
where $r < \mu_1 \leq \mu_2^* \ldots \leq \mu_{k}^*$ have been chosen earlier. Note that $u_k$ solves
\[
	\min\left(F_1^k(Du_k,D^2u_k),F_2^k(Du_k,D^2u_k)\right)\leq \left\|f_k\right\|_{L^\infty(B_1)}\quad\mbox{in}\quad B_1
\]
and
\[
	\max\left(F_1^k(Du_k,D^2u_k),F_2^k(Du_k,D^2u_k)\right)\geq -\left\|f_k\right\|_{L^\infty(B_1)}\quad\mbox{in}\quad B_1,
\]
where 
\[
	F_i^k(p,M)\coloneqq \sigma_i^k\left(\left|p+\frac{D\phi_{k-1}}{\mu^*_{k-1}}\right|\right)F_k(M),
\]
with    
\[
	\sigma_{i}^k(t) \coloneqq  \frac{\prod\limits_{j=1}^{k} \mu^*_j}{r^k} \sigma_i\left(\prod\limits_{j=1}^{k} \mu^*_j t\right),
\]
\[
	F_k(M) \coloneqq  r^k\left(\prod\limits_{j=1}^{k} \mu^*_j\right)^{-1} F\left(\left(\prod\limits_{j=1}^{k} \mu^*_j\right)\left(r^k\right)^{-1}M\right),
\]
 and 
\[
	f_k(x)\coloneqq f(r^k x).
\]
 
The choice of $(\mu^*_j)_{j=1}^k$ and the construction of the degeneracies $\sigma_i^k$ allows us to evoke Proposition \ref{prop_discretehyperplane} and obtain an affine function $\phi$ satisfying
    \begin{equation*}
         \sup_{x\in B_r} |u_{k}(x) -  \phi(x)| \leq \mu_1r.
    \end{equation*}
Now, rewriting the above in terms of $u$, we get
	\begin{equation*}
            \sup_{x\in B_{r^k}} |u(x) - \phi_k(x)| \leq \left(\prod\limits_{j=1}^{k} \mu^*_j \right) r^k,
        \end{equation*}
        where
	\begin{equation*}
            \phi_k(x) \coloneqq  \phi_1(x) + \sum\limits_{i=2}^{k-1} \phi_i(r^{-1}x)\prod\limits_{i=1}^{k-1} \mu_j^*  r^i = a_k + {\rm b}_k \cdot x.
        \end{equation*}
	In addition,
	\begin{equation*}\label{AK}
		|a_{k} - a_{k-1}| \leq C \left(\prod\limits_{j=1}^{k-1} \mu^*_j\right) r^{k-1},
	\end{equation*}
	and
	\begin{equation*}\label{BK}
		|{\rm b}_{k} - {\rm b}_{k-1}| \leq C \left(\prod\limits_{j=1}^{k-1} \mu^*_j\right),
        \end{equation*} 
        which completes the proof.
 \end{proof}

Next, we  present the proof of Theorem \ref{theo_regularity}, which stems from the choice of the sequence $(\mu_n)_{n\in\mathbb{N}}$ and the geometric decay in Proposition \ref{prop_osccontrol}.

\begin{proof}[Proof of Theorem \ref{theo_regularity}] 
We start by noticing two possibilities concerning the sequence $(\mu^*_n)_{n\in\mathbb{N}}$. Either the sequence repeats after some index $N \geq 2$ or we have $\mu_n^* < \mu^*_{n+1}$ for infinitely many indices $n\in\mathbb{N}$. 

In the former case, it is well-known that $C^{1,\alpha}$-regularity estimates are available, for $0<\alpha<\alpha_0$, where $\alpha_0\in(0,1)$ is the (universal) exponent associated with the Krylov-Safonov theory for $F=0$.
The second possibility amounts to 
    \begin{equation*}
	\frac{\prod\limits_{j=1}^{n+1}\mu_j^*}{ r^{n+1}}\sigma_i\left( \left[\prod\limits_{j=1}^{n+1}\mu^*_j \right] c_{n+1} \right) = 1. 
 \end{equation*}
Here, the definition of $\mu_n^*$ implies
   \begin{equation*}
	\frac{\prod\limits_{j=1}^{n}\mu^*_j}{r^{n}}\sigma_i\left(  \prod\limits_{j=1}^{n}\mu^*_j\,  c_{n} \right)= 1.
	\end{equation*}
	Hence,
    \begin{equation}\label{eq_taun}
	\prod\limits_{j=1}^{n}\mu^*_j = \frac{1}{c_{n}} \sigma_j^{-1}\left( \frac{r^{n}}{\prod\limits_{j=1}^{n}\mu^*_j} \right) \leq \frac{\sigma_i^{-1}(\theta_{n})}{c_{n}}.
    \end{equation}
Define $(\tau_n)_{n\in\mathbb{N}}$ as
    \begin{equation*}
	(\tau_n)_{n\in\mathbb{N}} \coloneqq  \left(\prod\limits_{j=1}^{n}\mu^*_j\right)_{n \in \mathbb{N}}.
	\end{equation*}
	
Notice that
\[
	\begin{split}
		\prod_{j=1}^{m+1} \mu_j^* &= \frac{1}{c_{m+1}} \sigma_i^{-1}\left(\frac{r^{m+1}}{\prod_{i=1}^{m+1} \mu_i^*} \right)  \leq  \frac{\sigma_i^{-1} \left( \theta^{m+1}  \right)}{c_{m+1}} \\&\leq 
		\frac{\max \left( \sigma_1^{-1} \left( \theta^{m+1}\right), \sigma_2^{-1} \left(\theta^{m+1} \right) \right)}{c_{m+1}}.
	\end{split}
\]	
Because of \eqref{eq_taun}, we conclude that $(\tau_n)_{n\in\mathbb{N}} \in\ell^1$, and its norm is bounded by 
\[
	\sum_{n=1}^{\infty}\max\left(\frac{\sigma_1^{-1}(\theta_{n})}{c_{n}},\frac{\sigma_2^{-1}(\theta_{n})}{c_{n}}\right)
\]
The latter is finite due to the summable characterisation of Dini continuity and Assumption \ref{assump_sigmabasic}. As a consequence, we infer that 
\[
	\lim_{n\to\infty}\left(\prod\limits_{k=1}^{n}\mu_k^*\right)=0;
\]
therefore, $(a_n)_{n\in\mathbb{N}}$ and $({\rm b}_n)_{n\in\mathbb{N}}$ are Cauchy sequences and there exists $a_\infty\in\mathbb{R}$ and ${\rm b}_\infty\in\mathbb{R}^d$ such that 
\[
	a_n\longrightarrow a_\infty\qquad\mbox{and}\qquad{\rm b}_n\longrightarrow{\rm b}_\infty,
\]  
as $n\to\infty$. Moreover, notice that 
    \begin{equation}\label{eq_stagecoach}
    	|a_{\infty} - a_n| \leq C\left(\sum\limits_{i=n}^{\infty}\tau_i\right) r^n \qquad \mbox{and} \qquad |{\rm b}_{\infty} - {\rm b}_n| \leq C\left(\sum\limits_{i=n}^{\infty}\tau_i\right).
    \end{equation}
Set $\varphi(x)\coloneqq a_\infty+{\rm b}_\infty\cdot x$ and fix $0<\rho\ll 1$. Let $n\in\mathbb{N}$ be such that $r^n<\rho\leq r^{n+1}$. Combine Proposition \ref{prop_osccontrol} with the inequalities in \eqref{eq_stagecoach} to obtain
\begin{equation}\label{eq_fofs}
	\begin{split}
		\sup_{x\in B_\rho}\left|u(x)-\varphi(x)\right|&\leq \sup_{x\in B_{r^n}}\left|u(x)-\phi_n(x)\right|+ \sup_{x\in B_{r^n}}\left|\varphi(x)-\phi_n(x)\right|\\
			&\leq \frac{1}{r}C\left(\tau_n+\sum_{i=n}^\infty\tau_i\right)\rho\\
			&\leq C\left(\sum_{i=n}^\infty\tau_i\right)\rho.
	\end{split}
\end{equation}
Because the parameter $n\in\mathbb{N}$ in the inequalities in \eqref{eq_fofs} is arbitrary, one can choose $n=n(\rho)$ as
\[
	n(\rho)\coloneqq \left\lfloor 1/\rho\right\rfloor.
\]
Define $\sigma:[0,\infty)\to[0,\infty)$ as
\[
	\sigma(t)\coloneqq \sum_{i=\left\lfloor 1/t\right\rfloor}^\infty\tau_i
\]
if $t>0$, with $\sigma(0)=0$. We conclude that $\sigma(t)$ is a modulus of continuity. Hence, for every $0<\rho\ll1$, \eqref{eq_fofs} becomes
\[
	\sup_{x\in B_\rho}\left|u(x)-\varphi(x)\right|\leq C\sigma(\rho)\rho,
\]
which completes the proof since $\varphi$ is an affine function.    
\end{proof}

\bigskip
	
{\small \noindent{\bf Acknowledgements.} This publication is based upon work supported by King Abdullah University of Science and Technology Research Funding (KRF) under Award No. ORFS-2024-CRG12-6430.3. EP is supported by the Centre for Mathematics of the University of Coimbra (CMUC, funded by the Portuguese Government through FCT/MCTES, DOI 10.54499/UIDB/00324/2020). DS is supported by the National Science Centre (NCN) Grant Sonata Bis 2019/34/E/ST1/00120.

\smallskip 

\noindent{\bf Declarations}

\smallskip

\noindent {Data availability statement:} All data needed are contained in the manuscript.

\noindent {Funding and/or Conflicts of interests/Competing interests:} The authors declare that there are no financial, competing or conflicts of interest.}

\bibliographystyle{plain}
\bibliography{biblio.bib}

\end{document}